\documentclass{article}

\usepackage{arxiv}
\usepackage[utf8]{inputenc}				
\usepackage[numbers,square,sort&compress]{natbib}
\usepackage{amsmath,amssymb}
\usepackage{float}
\usepackage[ruled,vlined,linesnumbered]{algorithm2e}
\usepackage{url}
\usepackage{mathtools,bm}
\usepackage{booktabs}
\usepackage{array}
\usepackage{enumitem}
\usepackage{hyperref}
\usepackage{cleveref}
\usepackage{graphicx}
\usepackage{subcaption}
\usepackage{amsthm}

\newtheorem{theorem}{Theorem}

\newcommand{\R}{\ensuremath{\mathbb{R}}}
\newcommand{\bu}{\ensuremath{\mathbf{u}}}
\newcommand{\bk}{\ensuremath{\mathbf{k}}}
\newcommand{\bg}{\ensuremath{\mathbf{g}}}
\newcommand{\ba}{\ensuremath{\mathbf{a}}}
\newcommand{\tba}{\ensuremath{\widetilde{\mathbf{a}}}}

\newcommand{\mriheading}[1]{%
  \begin{minipage}[t]{0.235\linewidth}
    \centering
    \normalfont $#1$
  \end{minipage}%
}

\newcommand{\mriimage}[1]{%
  \begin{minipage}[t]{0.235\linewidth}
    \centering
    \includegraphics[width=\linewidth]{#1}
  \end{minipage}%
}

\newcommand{\mriresult}[2]{%
  \begin{minipage}[t]{0.235\linewidth}
    \centering
    \includegraphics[width=\linewidth]{#1}\par
    \vspace{2pt}
    {\normalfont\scriptsize #2\par}
  \end{minipage}%
}

\title{Three-dimensional blind deconvolution by CP-parameterized kernels}

\date{} 					

\author{ Fatoumata Sanogo\\
	Department of Mathematics\\
	Bates College\\
	Lewiston, ME 04240 \\
	\texttt{fsanogo@bates.edu} \\
	\And
	Stefan Kindermann\\
	Industrial Mathematics Institute\\
	Johannes Kepler Universitat\\
	  Linz, Austria 4040 \\
	\texttt{kindermann@indmath.uni-linz.ac.at} \\
}

\renewcommand{\shorttitle}{\textit{arXiv} Template}

\hypersetup{
pdftitle={A template for the arxiv style},
pdfsubject={q-bio.NC, q-bio.QM},
pdfauthor={David S.~Hippocampus, Elias D.~Striatum},
pdfkeywords={First keyword, Second keyword, More},
}

\begin{document}
\maketitle

\begin{abstract}
We consider the problem of reconstructing a kernel $k$ and image $u$ in a blind deconvolution problem from convolutional data $g = k*u$. We particularly focus on the 3D case, i.e., on volumetric images. The problem is approached by imposing a semiparametric CP decomposition for the kernel and a variational TV penalty for the image. We also discuss a motion blur ambiguity effect, i.e., a nonuniqueness issue of the kernel in the blind deconvolution problem that may appear in certain videos of moving objects. The blind deconvolution algorithm uses an alternating minimization approach with TV regularization for $u$ and a CP decomposition step for the kernel. These are implemented with positivity projection, kernel normalization, and\textbackslash or causal support projection steps. Several numerical examples for 3D MRI images, hyperspectral images, and a grayscale test video show the effectiveness of the method.
\end{abstract}


\section{Introduction} 

Blind deconvolution refers to the problem of reconstructing the two components $k$ (kernel) 
and $u$ (image) in a convolution $g = k*u$ from possibly noisy data $g$. The problem has numerous 
applications in all fields related to image processing, such as astronomy, microscopy, medical imaging, and communication, but it also serves as a challenging nonlinear inverse benchmark problem in applied mathematics. 

One of the difficulties arises from the fact that deconvolution with a known kernel is itself an ill-posed problem, and with the kernel function---a multiplier in Fourier space---being an additional 
unknown, the troubles are also multiplied. In fact, the blind deconvolution problem suffers from severe nonuniqueness and nonlinearity in addition to ill-posedness; see, for instance, \cite{Chanbook,justen,xue2022b}. 

Existing approaches to blind deconvolution are often based on variational (or related Bayesian) 
approaches. Variational methods are often variants of general Tikhonov regularization, typically based on a least-squares minimization of the data fidelity together with penalty terms for $u$ and $k$ as priors. Practical methods follow from designing optimization routines to solve the regularization functional; see, for instance,   
\cite{BuOsh,justen,Osher2005,Chan,Wolf} for a non-exhaustive list of references.

Our approach is to consider the 3D deconvolution problem with a low-rank CP-decomposed kernel. That is, we suppose that the kernel tensor can be written as a finite sum of rank-one tensors. Representing the kernel as a low-rank object has been considered in the literature, but, to the best of our knowledge, this was not done for 3D kernels  with CP low rank. In \cite{si2019understanding}, the authors propose a low-rank-based regularization for a 2D blur kernel to exploit structural information in degraded kernels. Moreover, in \cite{perrone} it is shown that a low-rank SVD approximation of the PSF produces a short sum of Kronecker-product blurring operators. In \cite{ren2018deep}, a non-blind deconvolution, using generalized low-rank approximations of the pseudo-inverse kernels, is proposed, which  then leads to the representation of the  deconvolution as a weighted sum of separable one-dimensional filters.

The focus of this article is twofold. First, we consider blind deconvolution in three-dimensional space (instead of the more common two-dimensional kernel case). The data and unknowns $u$ and $k$ can be considered components of 3D tensors, where $u$ could represent grayscale videos, color images, or a family of images with an additional component factor. 

A second aspect, related to the above-mentioned 
nonuniqueness, is the use of parameterized kernels. Note that it is common to impose certain restrictions on the kernel to reduce the nonuniqueness problem; see~Section~\ref{sec:2}. Often, one parameterizes the kernel by a finite-dimensional family of functions, which obviously stabilizes the problem. Here, we use a semiparametric family, i.e., an infinite-dimensional class of functions, but parameterize the kernel through a tensor decomposition.
More specifically, we represent and compute kernels in the form of low-rank CP decompositions. We apply a variational approach with a TV-prior for $u$ and an implicit low-rank restriction 
for $k$. The resulting optimization procedure is solved by an alternating direction approach, with a CP decomposition replacing the minimization step for $k$. 
Furthermore, it has been pointed out in \cite{PerroneFavaro2} that several 
normalization and projection steps (and their order) are an important ingredient in blind deconvolution, and we include them as key components of our algorithm. 

Related to our focus on 3D images and videos, we also briefly discuss an apparent nonuniqueness, 
named motion blur uncertainty, in Section~\ref{sec:uncertainty} that might be relevant for videos of single moving objects. 

\paragraph*{Notation}
We use the following notation: Functions are 
represented by normal font lower case letters, 
while their discretization, i.e., tensors, matrices, and vectors are represented by bold-face lower-case letters ($\bf k$, $\bf a$, $\bf u$, etc.) The components of their entries are again normal-font letters with subindices, 
e.g., ${\bf k} = (k_{i,j,k})_{i,j,k}$, or, 
for brevity, we often simply write 
${\bf k} = k_{i,j,k}$.

\section{Blind deconvolution and CP decomposition}\label{sec:2}
Our abstract blind deconvolution problem is that of finding a blurring kernel $k$ and
a ``clean'' 3D image $u$ (or volumetric image)
from (possibly noisy) convolution data $g$ as follows:
\begin{equation}\label{main}
 k*u:= \int_{\R^3} k(x-y) u(y) dy = g(x) \qquad x,y \in \R^3. 
\end{equation}
See, for instance, \cite{Chan,justen} and the references therein for further context. 
As a minimal requirement we impose $k,u \in L^1(\R^3)$ such that  the convolution 
integral exists and is in $L^1(\R^3)$. For the kernel $k$, we always additionally impose 
a standard normalization assumption \cite[p.~223]{Chan}, namely 
\begin{align}\label{norm}
 &\int_ {\R^3} k(z) dz  = 1  \qquad \text{normalization}. 
\end{align}
Further constraints on $k$ that may or may not be used are 
\begin{alignat}{2} 
  k(z)&\geq 0&  \qquad &\text{positivity}, \\ 
  k(-z) &= k(z)&  \qquad &\text{symmetry}. \label{sym}\ 
\end{alignat}

A major problem with blind deconvolution is the inherent nonuniqueness. 
This can be formally seen by taking the Fourier transform, which renders the 
problem into 
\[ \hat{k}(\xi)\hat{u}(\xi) = \hat{g}(\xi), \qquad \xi \in \R^3, \]
from which it is obvious that the individual functions 
$\hat{k}$ and $\hat{u}$ are underdetermined in the 
product $\hat{k}\hat{u}$ given by the data.  The role of the normalization 
\eqref{norm} is to fix a multiplicative constant of $k$ in the product, while 
the symmetry \eqref{sym} makes $\hat{k}$ a real-valued function; thus, 
it fixes a phase factor in the product. Still, 
since we may write $\hat{k}\hat{u} = \frac{\hat{k}}{\hat{m}} \hat{u} \hat{m}$ 
with an arbitrary function $m(\xi) \not = 0$, the inherent nonuniqueness still remains. 

Besides the nonuniqueness, the ill-posedness due to instability of 
any typical deblurring problem raises even further problems. 
A well-known approach to tackling, though not resolving, these issues is 
to use regularization: A common variational regularization approach is 
to minimize a Tikhonov functional of the form 
\begin{equation}\label{tik}
 J(u,k) = \frac{1}{2}\| k*u - g\|_{L^1}^2 + \alpha R(u) + \beta S(k), 
\end{equation}
where $R$ and $S$ are respective penalty functionals for $u$ and $k$, respectively. 
Typically, for $R(u)$, the well-known total variation 
\[ R(u) = |u|_{TV} \]
is chosen, due to the popularity and success of the Rudin-Osher-Fatemi
\cite{RuOsFa} approach in 
denoising and deblurring. See, for instance, \cite{Chan}.  
Another approach is to use Sobolev spaces \cite{BuOsh} as penalties. 
Even though a regularization approach will result in some solution $k$, $u$
under weak regularity conditions, it is not guaranteed that this is 
``the'' wanted solution; moreover, the result also depends on the algorithm 
for minimizing the functional: the devil is in the detail
\cite{PerroneFavaro2}. 

\subsection{The semi-parameteric CP-model and the discretization} 
An alternative to the regularization/penalization approach is to reduce the complexity by imposing parametric models, in particular for $k$. 
That is, one a priori imposes the restriction that the kernel $k$ is out of some low/finite-dimensional manifold  \cite[Section 5.5.1]{Chan}. For instance, a set of Gaussians with varying variance could be used. This makes sense as long as the data are known to come from such a low-dimensional set of blurring operators. 

Our approach is a compromise in the sense that we keep a non-parameteric model for $k$ while reducing complexity by a CP decomposition. That is, we assume  that the kernel function can be written as a finite sum of tensor products of one-dimensional functions:
Thus, our continuous model assumption is 
\begin{equation}\label{contCP}
 k(x_1,x_2,x_3) = \sum_{r=1}^R a_{1,r}(x_1) a_{2,r}(x_2) a_{3,r}(x_3), \qquad x = (x_1,x_2,x_3)  \in \R^3,
\end{equation}
with $R\geq 1$ is a fixed integer. 

Although our main focus is the discrete model and 
its numerical computation, we  briefly discuss 
some analytical properties of the continuous 
Tikhonov regularization associated with the
semiparametric model. That is, consider the minimization of the  functional 
\begin{equation}\label{tik1}
\begin{split} 
 J(u,a_{1,r},a_{2,r},a_{3,r}) = \frac{1}{2}\| k*u - g\|_{L^1}^2 + \alpha |u|_{TV} &+ \beta 
 \left( \sum_{r=1}^R \|a_{1,r}\|_A^2 + 
 \|a_{2,r}\|_A^2 
 + \|a_{3,r}\|_A^2 \right) \\ 
& \text{subject to \eqref{contCP},}
\end{split}
\end{equation}
with some norm $\|.\|_A$ as regularization for 
the factors $a_{i,r}$.

Note that if an $L^1$-norm would be used as regularization for the kernel $k$, then the functional does not per se allow for a minimizer due to the non-reflexivity of $L^1$. Consequently, existence of a minimizer strongly depends  on the chosen norms for the factor components $a_{i,r}$. 
Let us state some sufficient conditions:

\begin{theorem}
Let $\alpha$, $\beta >0$, and let $R < \infty$.  
Assume that $u$ in \eqref{tik1} is restricted to 
a bounded domain and that the constraint 
$\int_{R^n} u(x)dx = c$, for some fixed $c$, 
is added to the minimization. 
Furthermore, assume that the penalty 
norms $\|.\|_A$ for the factors $a_{i,r}$ 
correspond to a Banach space that is 
compactly embedded into $L^1$. 
Then a minimizer of \eqref{tik1} exist. 
\end{theorem}
\begin{proof}
By the standard approach of variational analysis, 
we consider a minimizing sequence 
$(u_n,a_{i,r,n})$ 
that is hence bounded in the respective norms.
By compact embedding, we find an $L^1$-convergent subsequence of the factor sequence $a_{i,r,n} \to a_{i,r}$. Thus, the corresponding kernel $k_n$ 
converges as well in $L^1$ by Fubini's theorem and 
since the sum of the factors is finite. 
Also, by the added constraint on $u$, we may 
use a compact embedding of TV into $L^1$, 
which is valid in any dimension and for bounded domains \cite[Cor. 3.49]{Ambrosio}. Since the convolution 
is continuous from $L^1 \times L^1 \to L^1$, the existence of a minimizer follows. 
\end{proof}

In the remainder of the paper, we restrict ourselves to the discrete case. Here, the penalty  for the factor matrices is not needed  since 
we will impose a normalization 
$\int_{\R^n} a_{i,r}(x)dx = 1$ in the discrete case together with a positivity constraint, which thus implies an $L^1$-bound.  

Also the constraint for $u$ is not needed since this is  implicitly contained in the resulting condition $k* 1 \not = 0$; see \cite[Hypothesis (H2)]{Chambollelions}). 
It then follows  rather elementary that the 
discrete variant of \eqref{tik1} has a minimizer.

\subsubsection{The discrete model}
The corresponding discrete model is obtained after discretization. 
We restrict the domain for the convolution to $[-L,L]^3$, with $L$ large enough, 
and 
assume a uniform discretization on a 3D grid 
\begin{align*}
 x_{i,j,k} &= \left(-L +\tfrac{i-1}{(I-1) 2 L},-L +\tfrac{j-1}{(J-1) 2 L}, -L +\tfrac{k-1}{(K-1) 2 L}\right), \\
 &\qquad i=1,\ldots,I, \ j= 1,\ldots, J, \ k = 1,\ldots, K.
\end{align*} 
Let the components of the 3D-tensor ${\bf k} := k_{i,j,k}$ be given by the values  $k(x_{i,j,k})$,
and 
let ${\bf a}_{1,r} = a_{1,r}\left(-L +\tfrac{i-1}{(I-1) 2 L}\right)_{i=1,I}$  and similarly for ${\bf a}_{2,r}$ and ${\bf a}_{3,r}$.
Then the discrete version of  \eqref{contCP} is 
\begin{equation} k_{i,j,k} = {\bf k} := \sum_{r=1}^R {\bf a}_{1,r} \otimes  {\bf a}_{2,r} \otimes {\bf a}_{3,r},  \qquad 
{\bf a}_{1,r} = ({a}_{1,r})_i, 
{\bf a}_{2,r} = ({a}_{2,r})_j, 
{\bf a}_{3,r} = ({a}_{3,r})_k, 
\end{equation} 
where $\otimes$ denotes the tensor product of vectors. 
This identity is exactly the CP decomposition 
 of $k_{i,j,k}$ into $R$ rank-1 tensors.
In the same way, we discretize the volumetric image $u$ and represent it 
by the tensor $u_{i,j,k}$.  

Our blind deconvolution approach is now a Tikhonov-like regularization with 
the CP decomposition in place of a penalty for $k$. 
That is, after discretization of the convolution integral and using a discrete TV penalty 
for $u$, 
we end up  minimizing  
\begin{equation}\label{CP1}
\begin{split} 
 J_{CP}(\bu,{\bf a}_{1,r},{\bf a}_{2,r},{\bf a}_{3,r}) &= 
\frac{1}{2}\| {\bf k}*{\bf u} - {\bf g}\|_{\R^{I\times J\times K}}^2 + \alpha |{\bf u}|_{TV} \\
& \qquad \text{ s.t. } 
 {\bf k} = \sum_{r=1}^R {\bf a}_{1,r} \otimes  {\bf a}_{2,r} \otimes {\bf a}_{3,r}. 
  \end{split}
\end{equation}
Here the parameter $R$ is fixed and in our setup chosen as a small number, e.g., as $R = 1$. 
As discrete TV seminorm, we take the anisotropic version, 
\[ |{\bf u}|_{TV} = |D_1^+ {\bf u}|_{\ell^1} +
 |D_2^+ {\bf u}|_{\ell^1} +  |D_3^+ {\bf u}|_{\ell^1}, 
\]
where $D_i$, $i=1,2,3$, are the discrete forward difference operators with respect to the $i$th component of $\bu$ and $|.|_{\ell^1}$ denotes the discrete 
$L^1$-norm, i.e., the sum of absolute values.

In order to minimize the above functional, an alternating minimization  approach is 
used. This is appealing since it splits the problem into a standard deblurring one 
and an update step for the CP factors. 
Note that a minimization over $\bu$, given $ {\bf k} = \sum_{r=1}^R {\bf a}_{1,r} \otimes  {\bf a}_{2,r} \otimes {\bf a}_{3,r}$, 
is equivalent to solving 
\[  \frac{1}{2}\| {\bf k}*{\bf u} - {\bf g}\|_{\R^{I\times J\times K}}^2 + \alpha |{\bf u}|_{TV}  \to \min_{\bf u}, \]
which corresponds to TV-regularized variational deblurring, for which a variety 
of methods are available \cite{Chan, 5173518,tao, ZHANG2020107631,shao2024revisiting,sanogo2021}. 

\subsection{Motion-blur uncertainty}\label{sec:uncertainty}
We have pointed out some  well-known standard nonuniqueness issues with the blind deconvolution problem. 
During our numerical experiments, we encountered another nonuniqueness problem that 
is related to 3D images (videos) of moving objects and the related space/time  uncertainty in 
motion blur. This issue can be seen as a special case of shift uncertainty \cite[Theorem 5.9]{Chanbook}.

Let us illustrate this as follows: 
As stated above, we consider a 3D deconvolution problem with a low-rank CP decomposed kernel as in \eqref{contCP} and
without discretization.  Specifically, we now denote the components of a vector $x$ by $x = (x_1,x_2,t)$, indicating that the last variable corresponds to a time variable and 
$x_1$, $x_2$ are the 2D spatial coordinates. 
Assume now a CP-rank 1 decomposition of the kernel in  \eqref{contCP}with normalized components 
\[ k(x_1,x_2,t) = a_1(x_1) a_2(x_2) a_3(t),\qquad \int_\R a_i(x) dx = 1.  \]
The observed uncertainty is related to movies of moving objects; we thus model  $u$ by 
\[ u(x_1,x_2,t) = f(x_1,x_2- t),  \] indicating a movement in the $x_2$-direction. 
The recorded data according to model \eqref{main} are then 
\begin{equation}\label{m1} y(x_1,x_2,t) = \int_\R  \int_\R  \int_\R  a_1(x_1-y_1) a_2(x_2-y_2) a_3(t - \tau)  f(y_1,y_2- \tau) dy_1 dy_2 d\tau.  
\end{equation}
Upon the change of variable $ y_2 -  \tau \to \sigma$, i.e., $\tau =  y_2 -\sigma,$ 
we have 
\begin{equation}\label{m2} y(x_1,x_2,t) = 
 \int_\R  \int_\R  \int_\R 
 a_1(x_1-y_1) a_2(x_2-y_2) a_3(t -  y_2 +\sigma) f(y_1,\sigma) d \sigma dy_2 dy_1. \\
\end{equation} 
On the other hand, performing a change of variables in \eqref{m1} by 
$ y_2 -  \tau \to z_2$, i.e., $y_2  = z_2 +  \tau $ yields 
\begin{equation}\label{m3}  y(x_1,x_2,t) =  \int_\R  \int_\R  \int_\R  a_1(x_1-y_1) a_2(x_2-  \tau -z_2) a_3(t - \tau) f(y_1,z_2) 
 d \tau dz_2 dy_1.  
\end{equation}

To observe  the nonuniqueness, choose first a kernel that has only  blurring in the time-variable, i.e., a motion blur: 
\begin{equation*}   k_1:= k(x_1,x_2,t) = \delta(x_1) \delta(x_2) \psi(t),  
\end{equation*}
where $\delta$ represents the Dirac distribution.
In a second case, we consider a corresponding blur in the $x_2$-variable 
\begin{equation*}  k_2:= k(x_1,x_2,t) = \delta(x_1) \psi(-x_2) \delta(t).   
\end{equation*} 
An illustration of the different  factors of 
$k_1$ and $k_2$ is given in \cref{fig:fig1}.

The main point in the motion blur uncertainty is that both kernels yield the same data; 
hence, the blind deconvolution problem is ambiguous: 

Using $k_1$ in \eqref{m2}  
 yields  
\begin{align*}
  y:= k_1*u =  \int_\R \psi(t -  x_2 +\sigma) f(x_1,\sigma) d \sigma  
  \end{align*}
while using  $k_2$ in \eqref{m3} yields   
\begin{align*}
  y:= k_2*u &=  \int_\R  \psi(-(x_2-  t -z_2))  f(x_1,z_2) dz_2  \\
  &= 
  \int_\R  \psi(t -x_2 + z_2) f(x_1,z_2) dz_2  =   k_1*u .
  \end{align*}
Thus, for a moving object, a temporal blur and a corresponding spatial blur 
(in the opposite movement direction) are indistinguishable.  By linearity, 
a similar  issue can appear in a rank-$R$ decomposition \eqref{contCP}.

Note that this effect is most pronounced when there is 
a single object
moving in the frame with constant velocity, while it might disappear for 
multiple moving objects or with irregular movement.  
We also mention that we observed this effect only in a few numerical experiments and only with severe over-regularization for large noise. In practically relevant cases, it was 
not reproducible. The reason might be that for variational regularization  in a low-noise (or rather no-noise-limit) case, the reconstructed kernel does not converge to the "ground truth", as observed by Perrone and Favaro \cite{PerroneFavaro1}; see also 
\cite{Benichoux}.

\begin{figure}[H]
\begin{minipage}{0.49\textwidth} 
\centering
\includegraphics[width=0.9\textwidth]{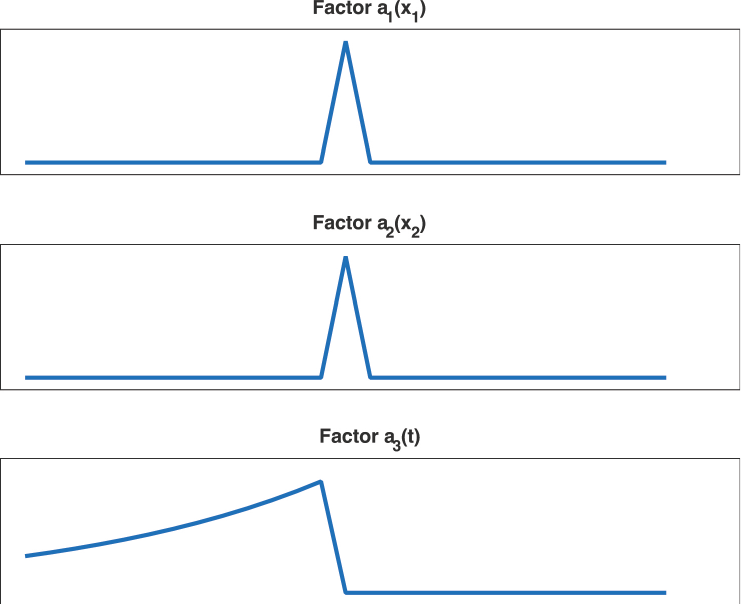}\\
{\footnotesize Factors of kernel $k_1$}
\end{minipage} 
\begin{minipage}{0.49\textwidth} 
\centering 
\includegraphics[width=0.9\textwidth]{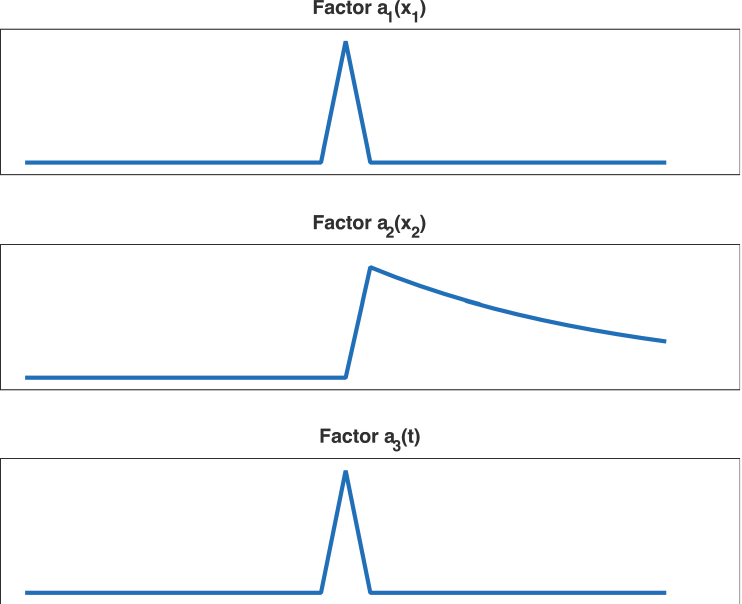}\\
{\footnotesize Factors of kernel $k_2$.} 
\end{minipage} 
\caption{Illustration of the uncertainty in the kernel factors. Left: the factors $a_1,a_2,a_3$ for the 
first kernel $k_1 = a_1(x_1) a_2(x_2) a_3(x_3)$.  
Right: alternative factors for the second kernel 
$k_2=a_1(x_1) a_2(x_2) a_3(x_3)$ yielding the same 
data for the moving object example in Section~\ref{sec:uncertainty}. 
} \label{fig:fig1}
\end{figure} 

\section{The algorithm}
In this  section, we describe the algorithm for solving the 3D blind deconvolution problem.

 As discussed above, instead of estimating all entries of $\bf k$ independently, the kernel is represented by a rank-\(R\) canonical polyadic decomposition:
 Let us define the matrices of factors 
\[\tba_{1} = [\ba_{1,1} \cdots \ba_{1,R}]\in \mathbb R^{I\times R},\quad
\tba_{2} =  [\ba_{2,1} \cdots \ba_{2,R}]\in \mathbb R^{J\times R},\quad
    \tba_{3} =  [\ba_{3,1} \cdots \ba_{3,R}]\in \mathbb R^{K\times R}.\]

To solve \cref{CP1}, we rely on an alternating minimization algorithm (see Algorithm \ref{alternate}), where the sharp 3D image $\bu$ and 3D kernel $\bk$ are updated alternately. In Algorithm~\ref{alternate}, we experimentally use a  value of $b\in[0.900,1].$
The expression $\operatorname{prox}_{\alpha \mathrm{TV}}$ denotes the proximal map with respect 
to the functional $\alpha |.|_{TV}$;
see, e.g., \cite{Beckbook}.

\begin{algorithm}[H]
\caption{Blind Deconvolution Algorithm}
\label{alternate}

\KwData{Observed blurred input $\bg$, blur size, initial $\alpha$, final $\alpha_{\min}$}
\KwResult{Reconstructed output $\bu$, estimated kernel $\bk$}

$\bu^0 \leftarrow \mathrm{pad}(\bg)$\;
$\bk^0 \leftarrow \mathrm{CP\_initialize}(\bk_x,\bk_y,\bk_z)$\;

\While{not converged}{
   $\bu^{t+1} \leftarrow 
    \operatorname{prox}_{\alpha \mathrm{TV}}
    \left(
    \bu^t -
    \epsilon_u \nabla_{\bu}
    \frac{1}{2}
    \|\bu^t * \bk^t - \bg\|_2^2
    \right)$\; \label{uupdate} 

    $\tba_{1},\tba_{2},\tba_{3} \leftarrow 
    \arg\min_{\ba_{1},\ba_{2},\ba_{3}}
    \frac{1}{2}
    \|\bu^{t+1} * \bk(\tba_{1},\tba_{2},\tba_{3}) - \bg\|_2^2$\;\label{kupdate1}

    $\bk^{t+1/2} \leftarrow 
    \displaystyle\sum_{r=1}^R {\bf a}_{1,r} \otimes  {\bf a}_{2,r} \otimes {\bf a}_{3,r}$;
\label{kupdate2} 
    $\bk^{t+1/2} \leftarrow \max\{\bk^{t+1/2},0\}$\;

    $\bk^{t+1} \leftarrow 
    \displaystyle\frac{\bk^{t+1/2}}{\|\bk^{t+1/2}\|_1}$\;

    $\alpha \leftarrow \max\{b\alpha,\alpha_{\min}\}$\;
}

$\bu \leftarrow \bu^{t+1}$\;
$\bk \leftarrow \bk^{t+1}$\;

\end{algorithm}

Now, let us describe the main steps of the algorithm in more detail.
\paragraph*{Image Update---One Proximal-Gradient Step}
The update of $\bu$ in line~\ref{uupdate} refers to the following procedure:

For a fixed kernel ${\bf k} = \sum_{r=1}^R {\bf a}_{1,r} \otimes  {\bf a}_{2,r} \otimes {\bf a}_{3,r}$, the sharp image is updated by approximately solving:
\begin{equation}\label{eq:image_subproblem}
\min_u \ \frac{1}{2}\| {\bf k}*\bu - \bg\|_F^2 + \alpha |\bu|_{TV}, 
\end{equation}
 The algorithm does not fully solve \eqref{eq:image_subproblem} at each outer iteration. Instead, it performs several proximal-gradient steps. We noticed that depending on the input,  it may be necessary to perform a couple of image update for each kernel update in our alternating minimization algorithm. Furthermore, we found that iteratively decreasing the value of $\alpha$ as shown in Algorithm \ref{alternate} improves the reconstruction.

 \paragraph*{CP-parameterized kernel update.}
The second main ingredient in Algorithm \ref{alternate} is the kernel update in lines~\ref{kupdate1}--\ref{kupdate2}:

For a fixed sharp image $\bu$,  \cref{CP1} becomes:

\begin{equation}\label{kernel-opts}
\begin{split} 
 J_{CP}(u,{\bf a}_{1,r},{\bf a}_{2,r},{\bf a}_{3,r}) &= 
\frac{1}{2}\| {\bf k}*\bu - \bg\|_{F}^2\\
& \qquad \text{ s.t. } 
 {\bf k} = \sum_{r=1}^R {\bf a}_{1,r} \otimes  {\bf a}_{2,r} \otimes {\bf a}_{3,r}. 
  \end{split}
\end{equation}

 \noindent At each outer iteration in 
 Algorithm~\ref{alternate}, the sharp image $\bu$ is fixed and the blur kernel is updated. The main assumption here is that the blur kernel admits a low-rank CP representation. This reduces the number of unknowns in the kernel update and enforces a structured, low-dimensional kernel model.

The discrete kernel tensor ${\bf k}$ is 
normalized analogously to \eqref{norm}: 
Let 
\begin{equation}
    \langle \bk, \mathbf{1} \rangle
=
\sum_{i=1}^{I}\sum_{j=1}^{J}\sum_{k=1}^{K} {k}_{ijk},
    \label{eq:normalized-cp-kernel}
\end{equation}
where \(\mathbf 1\) is the all-ones tensor of size \(I\times J\times K\). 
Then the normalization reads
\[
 \langle \bk, \mathbf{1} \rangle =   \sum_{i,j,k} k_{ijk}=1.
\]

Denote the gradient of $J_{CP}(\bu,\tba_{1},
\tba_{2},\tba_{3})$ with respect to the normalized kernel
$\bf k$ by
\[
    \mathcal H
    =
    \nabla_{\bf k}J_{CP}(\bu,\tba_{1},\tba_{2},
    \tba_{3}),
\]
where ${\bf k}$ is normalized as in \cref{eq:normalized-cp-kernel}.
Using the mode-$n$ matricization of the CP tensor, the gradients with respect to
the factor matrices are
\begin{align}
    \nabla_{\tba_a} J_{CP}
    &=
    \mathcal H_{(1)} (\tba_3 \odot \tba_2),
    \label{eq:grad_a1}
    \\
    \nabla_{\tba_2} J_{CP}
    &=
    \mathcal H_{(2)} (\tba_3 \odot \tba_1),
    \label{eq:grad_a2}
    \\
    \nabla_{\tba_3} J_{CP}
    &=
    \mathcal H_{(3)} (\tba_2 \odot \tba_1),
    \label{eq:grad_a3}
\end{align}
where $\odot$ denotes the Khatri--Rao product and
$\mathcal H_{(n)}$ denotes the mode-$n$ unfolding of $\mathcal H$.
Algorithm \ref{alternate} performs one alternating gradient step in factor space:
first \(\tba_1\) is updated while \(\tba_2\) and \(\tba_3\) are fixed, then \(\ba_2\) is updated,
and finally \(\tba_3\) is updated. After each block update, an optional positivity projection is applied to the factors, and the CP columns may be rebalanced to improve numerical conditioning. A final global normalization step is then
performed so that the reconstructed kernel again satisfies
\(\sum_{i,j,k} k_{ijk}=1\).

\section{Implementation details}

In Algorithm \ref{alternate}, at each iteration we perform just one CP decomposition update to find $\bk$ and $n$  proximal total variation steps on $\bu$. Depending on the blurred input $\bg$, the value of $n$ can be as low as $1$. We do this because we experimentally observed that the kernel tends to converges faster than the blurred input $\bg$ in the alternating minimization framework (see Algorithm \ref{alternate}. 

\noindent We use the notation $*$ to denote the discrete convolution operator, where the output is computed only on the valid region; i.e., if $\bg = \bu * \bk$, with $\bk \in \mathbb{R}^{I \times J \times K}$ and $\bu \in \mathbb{R}^{m \times n \times o}$, then we have
$\bg \in \mathbb{R}^{(m-I+1)\times(n-J+1) \times (o-K+1)}$. As shown in \cite{PerroneFavaro2}, the proper use of the convolution operator, as well as the correct choice of the domains of $\bu$ and $\bk$, improves the performance of the algorithm. Hence, in our algorithm, we consider the domain of the sharp image $\bu$ to be larger than the domain of the blurry image $\bg$. If $\bk \in \mathbb{R}^{I \times J \times K}$ and $\bu \in \mathbb{R}^{m \times n \times o}$, then
$\bg \in \mathbb{R}^{(m-I+1)\times(n-J+1) \times (o-K+1)}.$ 

\subsection{Rebalancing.} 
A CP decomposition has a scaling indeterminacy (see \cite{koldanet}). This means that  we can scale the individual vectors, i.e.,
\[
    \ba_{1,r}\otimes \ba_{2,r}\otimes \ba_{3,r}
    =
    (\lambda_r \ba_{1,r})\otimes(\gamma_r \ba_{2,r})\otimes\left(\beta_r \ba_{3,r}\right).
\] as long as \(\lambda_r\beta_r\gamma_r = 1.\)

 Hence, the same rank-one tensor can be represented by factors with highly unbalanced norms. To improve numerical stability, our algorithm rebalances the three vectors in each rank-one component. This is done as follows:
for nonzero individual factor-vectors norms, the vectors are rescaled by 
\begin{equation}
    \ba_{1,r}\leftarrow \frac{g_r}{\|\ba_{1,r}\|_2}\ba_{1,r},
    \qquad
    \ba_{2,r}\leftarrow \frac{g_r}{\|\ba_{2,r}\|_2}\ba_{2,r},
    \qquad
    \ba_{3,r}\leftarrow \frac{g_r}{\|\ba_{3,r}\|_2}\ba_{3,r}.
    \label{eq:cp-rebalancing}
\end{equation}
where $g_r =\left(\|x_r\|_2\|y_r\|_2\|z_r\|_2\right)^{1/3}.$
This step is optional in the implementation, as we learned from experience that, depending on the input, it might not be needed for a good numerical output. However, it can improve output quality significantly if used.

\subsection{Relaxation.} Recall that at each kernel-update step, the algorithm updates the factor matrices sequentially. For any factor \(F\in\{\tba_1,\tba_2,\tba_3\}\) update, we define an adaptive step size
\begin{equation}
    \eta_F
    =
    \frac{\gamma}{\max\{\|\nabla_F J_{CP}\|_F,\varepsilon\}},
    \label{eq:block-stepsize}
\end{equation}
where \(\gamma>0\) is a step-scale parameter and \(\varepsilon>0\) is a small stabilization constant.  With relaxation parameter \(\rho\in(0,1]\), the block update is then
\begin{equation}
    F^{+}
    =
    (1-\rho)F
    +
    \rho\left(F-\eta_F\nabla_F J_{CP}\right).
    \label{eq:relaxed-block-update}
\end{equation}
The relaxation form is useful because it allows the gradient step to be damped without changing the computed gradient direction. It has also been proven  useful in the alternating minimization setting; see \cite{cortild2025krasnoselskii, wen2012solving, neal2011distributed} for more details. Applying relaxation to the kernel update significantly improved  numerical reconstruction. 

\subsection{Positivity Projection and Final Kernel Normalization.}

In the kernel update step, the updated factors are projected onto the nonnegative orthant:
   \[ \ba_j\leftarrow \max\{\ba_j,0\},\]
where $j=1,2,3.$
After the projection, the CP tensor is rebuilt, and   the reconstructed kernel is then normalized using \cref{eq:normalized-cp-kernel}.

\subsection{Causal support projection.} 

Recall that the kernel is represented by the CP decomposition ($\bk=\sum_{r=1}^{R}\ba_{1,r}\otimes \ba_{2,r}\otimes \ba_{3,r}$), where the factors correspond to the three data dimensions \cite{koldanet}. We added an option for support constraints that can incorporate known acquisition physics. For instance, if the third dimension represents time and the blur is known to be causal, we set a `causal\_len' which is the number of temporal samples retained in each third-mode factor ($\ba_{3,r}$). This implies that the entries of ($\ba_{3,r}$) beyond this causal window are set to zero, restricting the kernel to present and past frames~\cite{5206815}. This reduces physically impossible solutions and shift uncertainty. Setting `causal\_len = 0' disables the constraint and allows unrestricted or symmetric kernels, such as Gaussian-like blur, to be estimated blindly.

\subsection{Non-blind error ratio}\label{eval}

Kernel quality is summarized by the \textbf{error ratio} described in~\cite{5206815}. We run a \emph{non-blind} TV deconvolution with the estimated kernel and, separately, with the true kernel, using the same solver settings, and then report
\begin{equation}
  \rho \;=\; \frac{\mathrm{MSE}\big(\text{recon with }K_{\mathrm{est}}\big)}
                  {\mathrm{MSE}\big(\text{recon with }K_{\mathrm{true}}\big)}.
  \label{eq:ratio}
\end{equation}
A ratio $\rho\approx 1$ means that the estimated kernel reconstructs almost as well as the true kernel. This diagnostic is run to convergence and is independent of the blind solver's stopping criteria.

\subsection{Stopping criteria}

For our algorithm, at iteration ($t$), the root-mean-square residual is computed as follows: $$\mathrm{rms}_t=||\bk_t*\bu_t-\bg||_2/\sqrt{N},$$ where $\bk_t$ and $\bu_t$ are the estimated kernel and image, respectively, at iteration $t$, and $\bg$ is the observed blurry/noisy image. 
We then follow  Morozov’s discrepancy principle \cite{morozov2012methods} 
accordint to 
which the algorithm stops at the first iteration satisfying ($\mathrm{rms}_t\leq\hat{\sigma}$), since reducing the residual below the estimated noise level may fit noise rather than recover additional signal. If this condition is not reached within the iteration budget, the final iterate is reported.
To ensure that our run remains blind we estimate the noise standard deviation directly from the observation ($\bg$). We use the `NoiseLevel' method described in \cite{6607209}. This is applied independently to each two-dimensional slice, and the median of the valid slice estimates defines the noise estimate ($\hat{\sigma}$). 

\section{Numerical experiments}

To test the performance of our algorithm, we have three types of experiments. In all simulations, final $\alpha_{min} = 0.0001$ and initial $\alpha = 0.02$. All experiments were run assuming a low-rank kernel and we used rank $R=1$. We evaluated the estimated kernel by taking the ratio between the nonblind deconvolution error with the estimated kernel and the nonblind deconvolution with the true kernel; see section~\ref{eval}. All experiments were conducted using the programming language Matlab.

\subsection{Case 1} The first experiment was performed on an MRI image of size $128\times 128 \times 27$ from the Matlab library. Our blind experiment was run using a synthetic Gaussian point spread function (PSF) with support $11\times 11\times 5$. \Cref{fig:mri}
shows different slices of the MRI image. To show the performance of our method we report PSNR, SSIM and error ratio ($\rho$) values in \cref{fig:mri}.

\begin{figure}[H]
    \centering

    \makebox[\linewidth][c]{%
        \mriheading{z=14}\hfill
        \mriheading{z=10}\hfill
        \mriheading{z=3}\hfill
        \mriheading{z=20}%
    }

    \vspace{3pt}

    \makebox[\linewidth][c]{%
        \mriimage{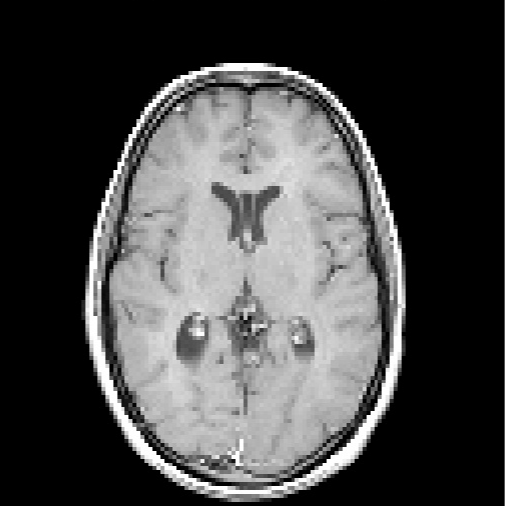}\hfill
        \mriimage{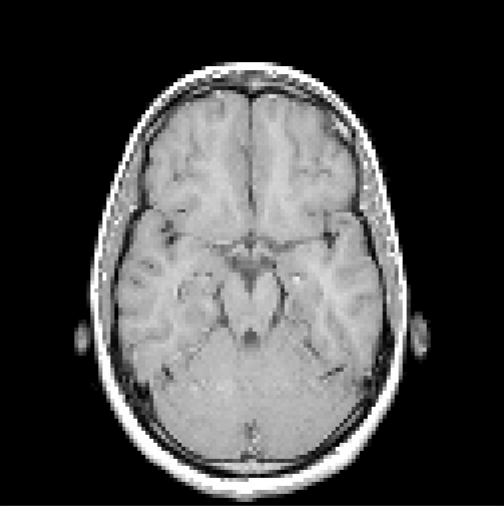}\hfill
        \mriimage{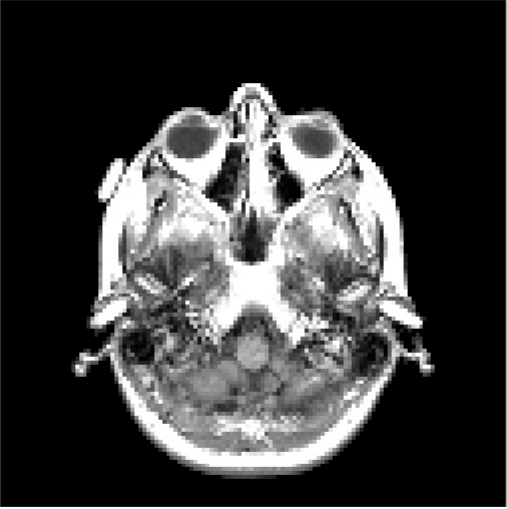}\hfill
        \mriimage{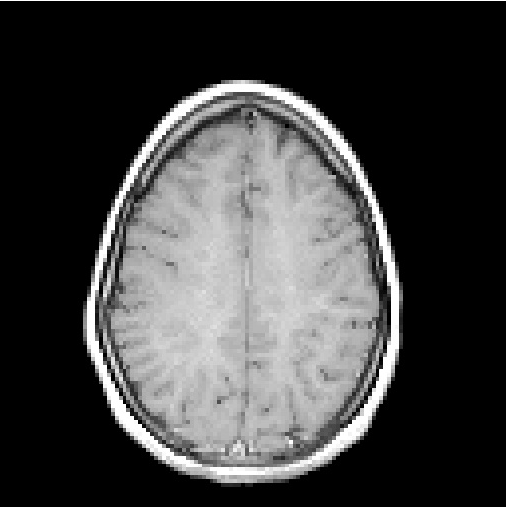}%
    }

    \vspace{6pt}

    \makebox[\linewidth][c]{%
        \mriresult{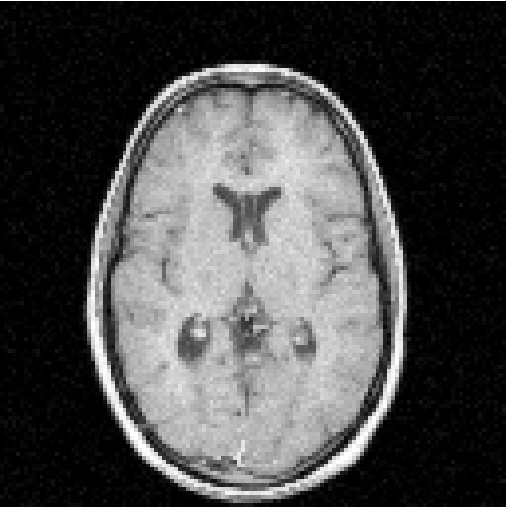}{%
            SSIM = 0.92541\\
            PSNR = 36.93 dB%
        }\hfill
        \mriresult{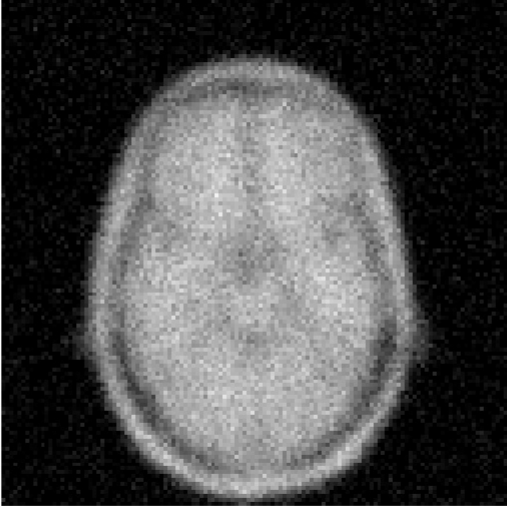}{%
            SSIM = 0.61295\\
            PSNR = 26.72 dB%
        }\hfill
        \mriresult{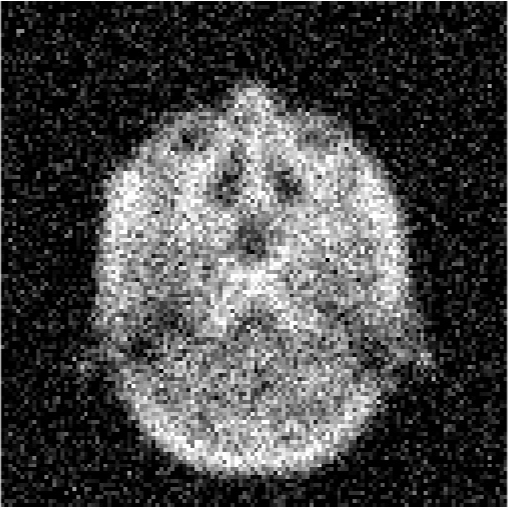}{%
            SSIM = 0.26115\\
            PSNR = 22.78 dB%
        }\hfill
        \mriresult{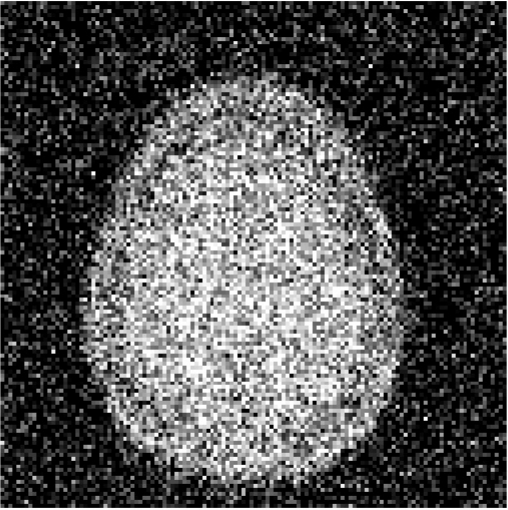}{%
            SSIM = 0.14250\\
            PSNR = 19.32 dB%
        }%
    }

    \vspace{6pt}

    \makebox[\linewidth][c]{%
        \mriresult{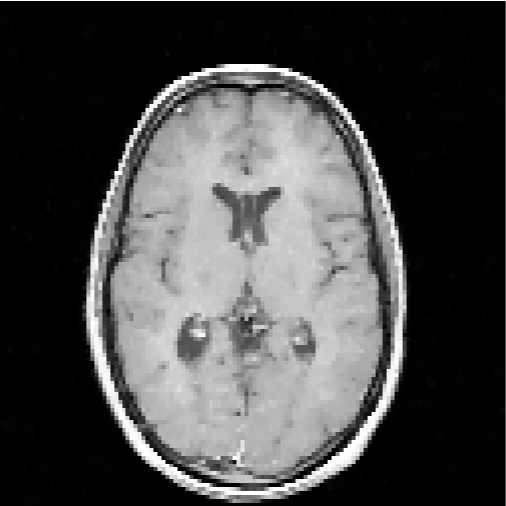}{%
            SSIM = 0.97643\\
            PSNR = 39.31 dB\\
            $\rho = 1.214$%
        }\hfill
        \mriresult{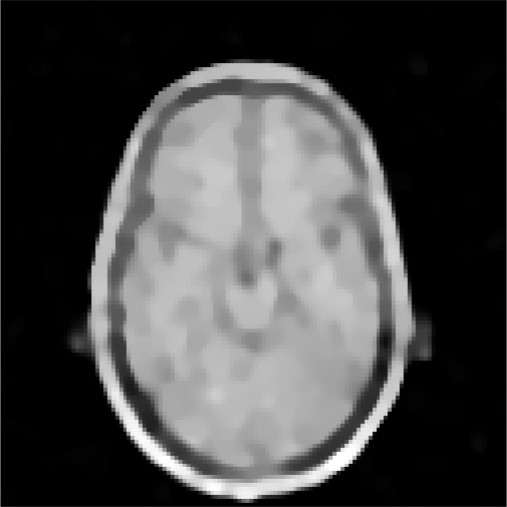}{%
            SSIM = 0.85631\\
            PSNR = 29.09 dB\\
            $\rho = 1.019$%
        }\hfill
        \mriresult{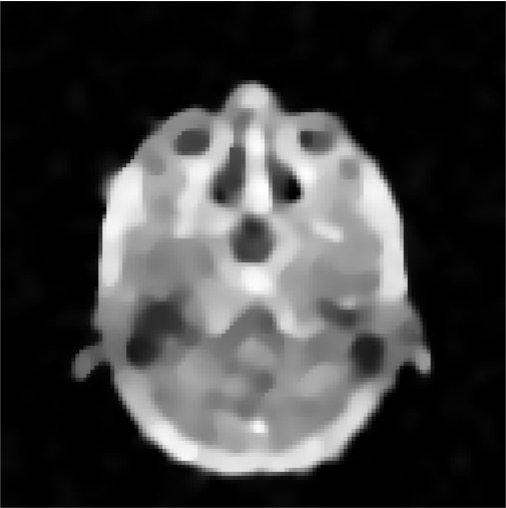}{%
            SSIM = 0.76163\\
            PSNR = 27.41 dB\\
            $\rho = 1.036$%
        }\hfill
        \mriresult{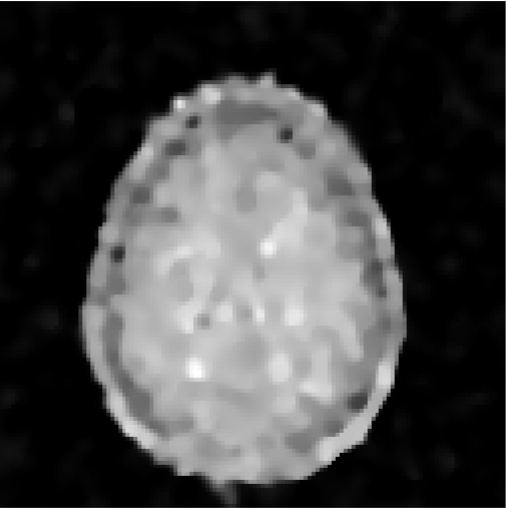}{%
            SSIM = 0.70624\\
            PSNR = 27.01 dB\\
            $\rho = 1.029$%
        }%
    }

    \caption{Representative axial slices from the 3D MRI volume.
    The $z$-values above indicate the specific slices shown.
    The first row shows the ground-truth slices, the second row shows the blurred and noisy observations with their SSIM and PSNR values, and the third row shows the reconstructed slices with their SSIM, PSNR, and kernel error ratio~($\rho$) values.}

    \label{fig:mri}
\end{figure}

\subsection{Case 2}

For this second experiment, we used an hyperspectral image, Samson \cite{zhu2017hyperspectral} of size $95\times 95\times 156$. We selected the first 78 bands for our experiment. A synthetic Gaussian kernel with support size 11 by 11 by 7 was used. To show the performance of our method we report PSNR, SSIM and error ratio ($\rho$) values in \cref{fig:hsi}.

\begin{figure}[H]
    \centering

    \makebox[\linewidth][c]{%
        \mriimage{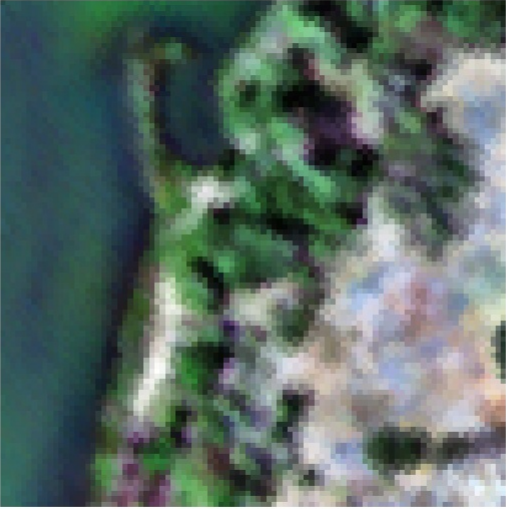}\hfill
        \mriimage{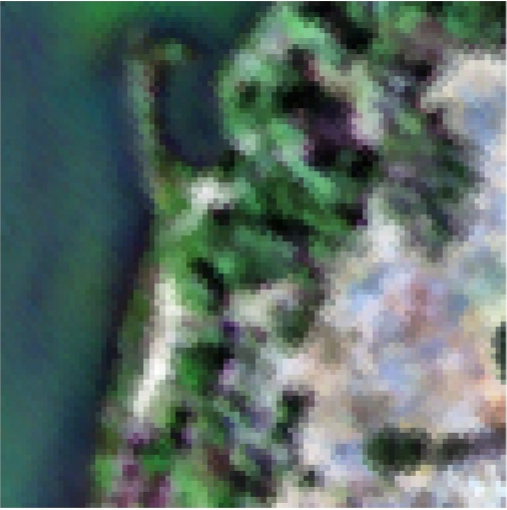}\hfill
        \mriimage{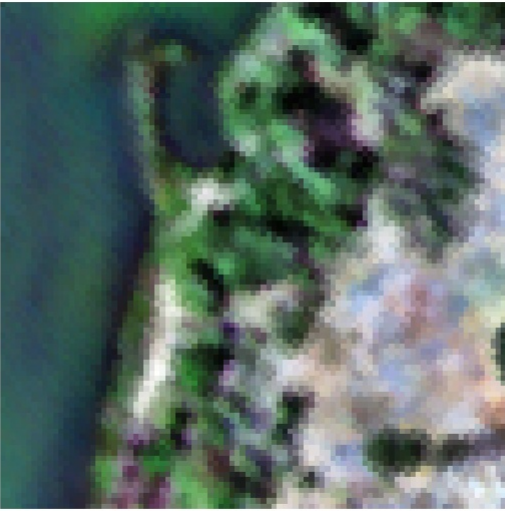}\hfill
        \mriimage{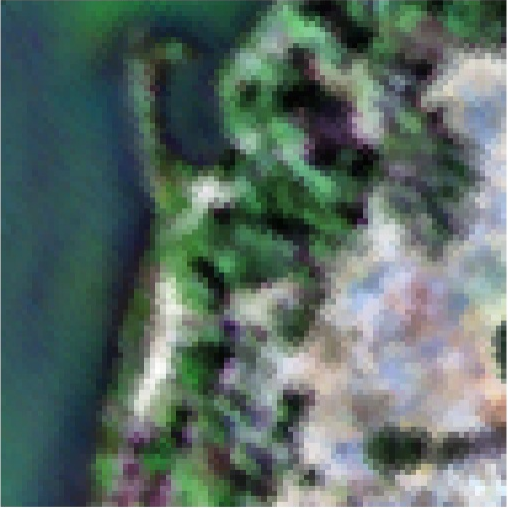}%
    }

    \vspace{6pt}

    \makebox[\linewidth][c]{%
        \mriresult{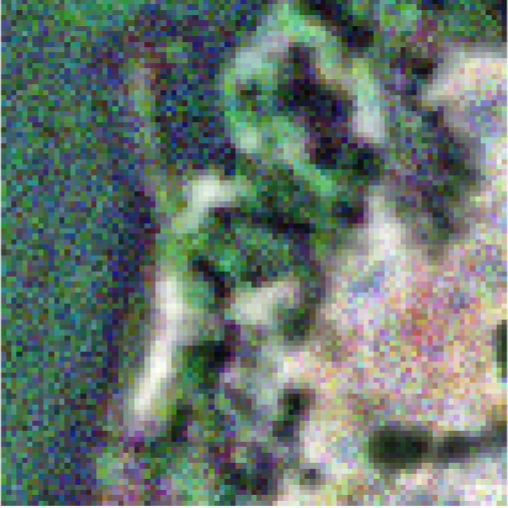}{%
            SSIM = 0.80136\\
            PSNR = 33.41 dB%
        }\hfill
        \mriresult{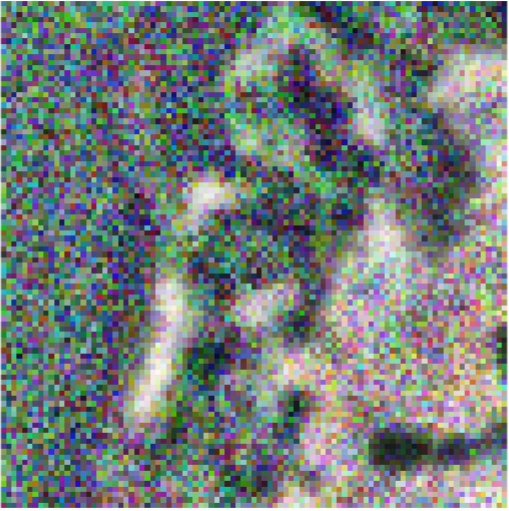}{%
            SSIM = 0.42484\\
            PSNR = 25.72 dB%
        }\hfill
        \mriresult{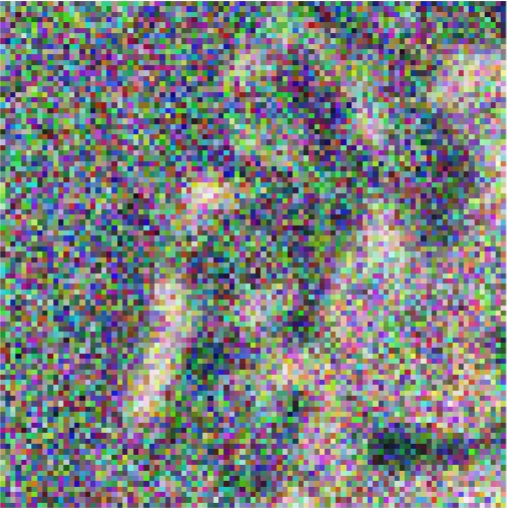}{%
            SSIM = 0.17431\\
            PSNR = 19.86 dB%
        }\hfill
        \mriresult{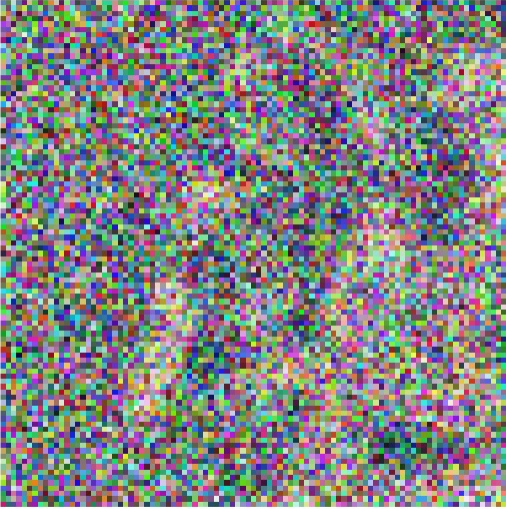}{%
            SSIM = 0.05508\\
            PSNR = 13.93 dB%
        }%
    }

    \vspace{6pt}

    \makebox[\linewidth][c]{%
        \mriresult{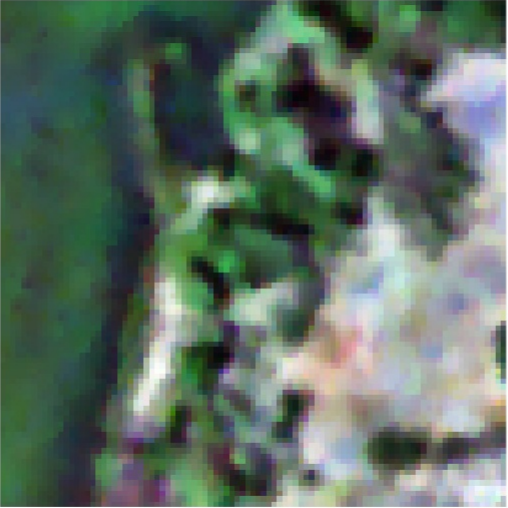}{%
            SSIM = 0.96579\\
            PSNR = 39.12 dB\\
            $\rho = 1.022$%
        }\hfill
        \mriresult{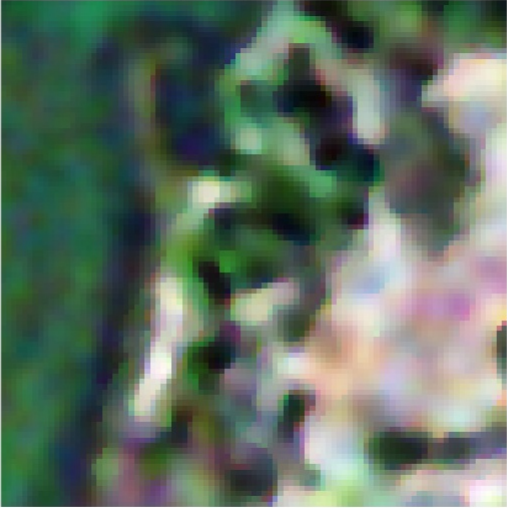}{%
            SSIM = 0.92269\\
            PSNR = 34.91 dB\\
            $\rho = 1.076$%
        }\hfill
        \mriresult{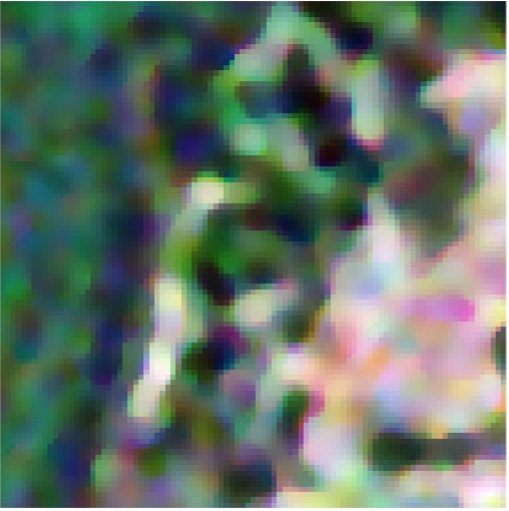}{%
            SSIM = 0.87479\\
            PSNR = 32.67 dB\\
            $\rho = 1.011$%
        }\hfill
        \mriresult{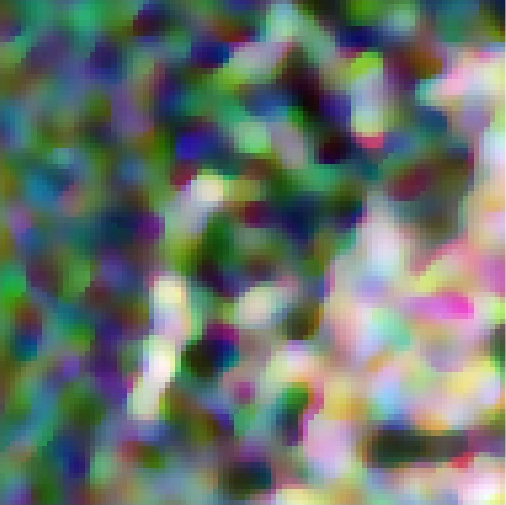}{%
            SSIM = 0.78371\\
            PSNR = 30.09 dB\\
            $\rho = 1.012$%
        }%
    }

    \caption{Hyperspectral image results. The first row presents the
    ground-truth images, the second row presents the blurred and noisy
    observations with their SSIM and PSNR values, and the third row
    presents the reconstructed images with their SSIM, PSNR, and kernel
    error ratio~($\rho$) values.}

    \label{fig:hsi}
\end{figure}

\subsection{Case 3}

Our third experiment is a short grayscale video (Matlab's shuttle.avi, 118 frames) treated as a three-dimensional 
object (two spatial dimensions and one temporal dimension) that was degraded using a synthetic motion blur \cref{fig:shuttle}. The video was cropped to a size $144\times 128\times 59$. To show the performance of our method, we report PSNR, SSIM and error ratio ($\rho$) values in \cref{fig:shuttle}. 

\begin{figure}[H]
    \centering

    \makebox[\linewidth][c]{%
        \mriheading{z=30}\hfill
        \mriheading{z=40}\hfill
        \mriheading{z=50}\hfill
        \mriheading{z=15}%
    }

    \vspace{3pt}

    \makebox[\linewidth][c]{%
        \mriimage{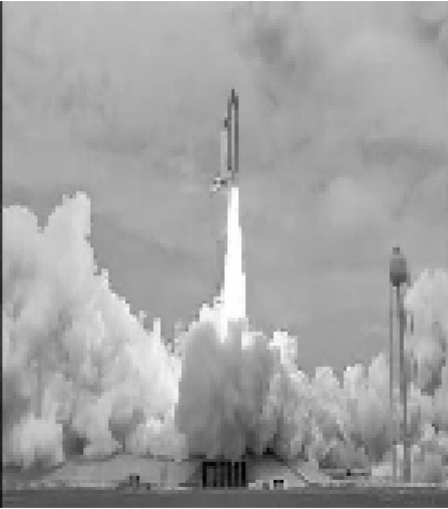}\hfill
        \mriimage{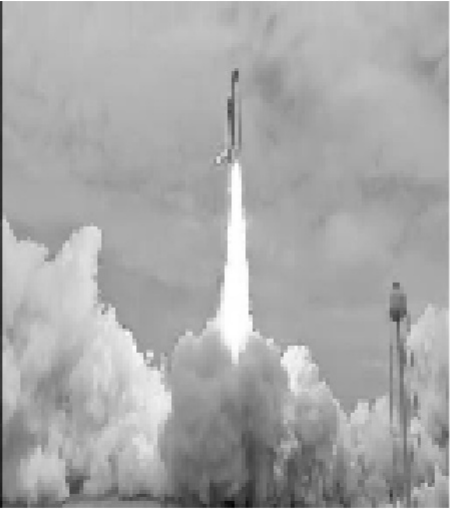}\hfill
        \mriimage{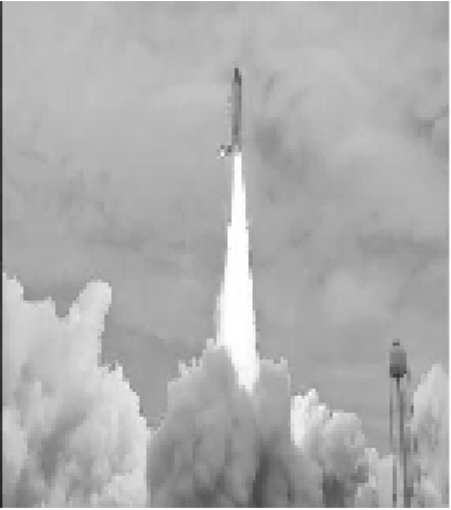}\hfill
        \mriimage{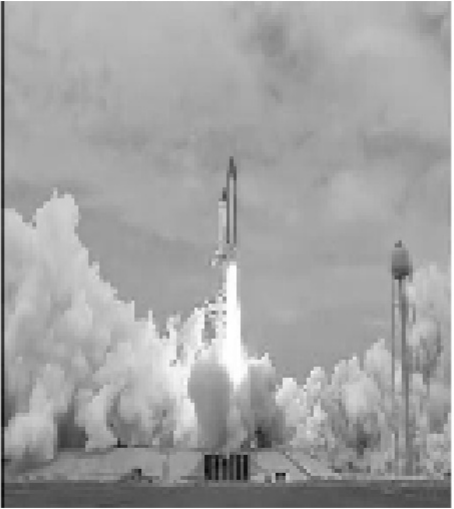}%
    }

    \vspace{6pt}

    \makebox[\linewidth][c]{%
        \mriresult{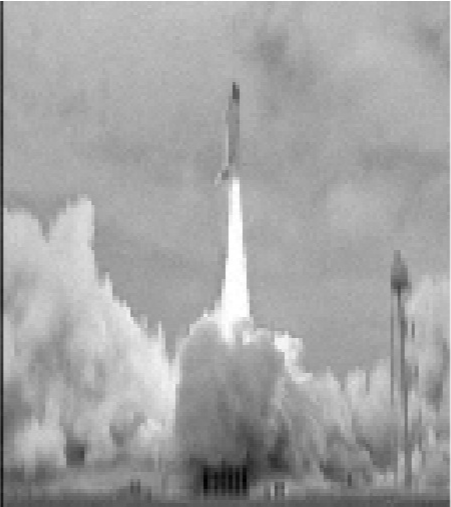}{%
            SSIM = 0.80207\\
            PSNR = 28.04 dB%
        }\hfill
        \mriresult{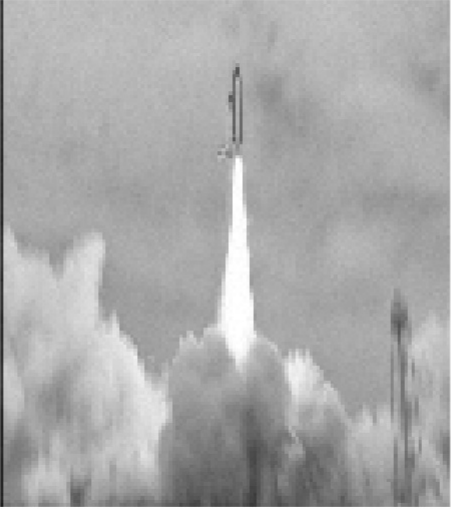}{%
            SSIM = 0.80207\\
            PSNR = 28.04 dB%
        }\hfill
        \mriresult{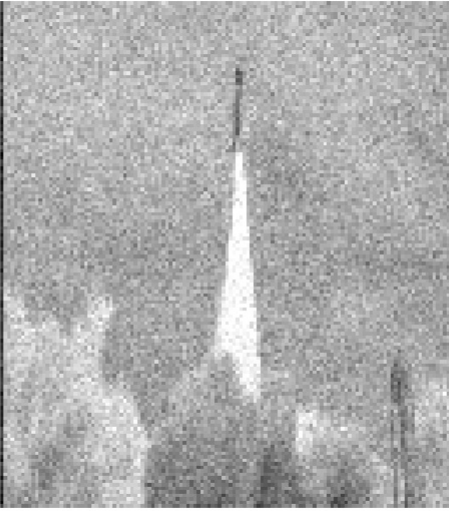}{%
            SSIM = 0.33196\\
            PSNR = 22.45 dB%
        }\hfill
        \mriresult{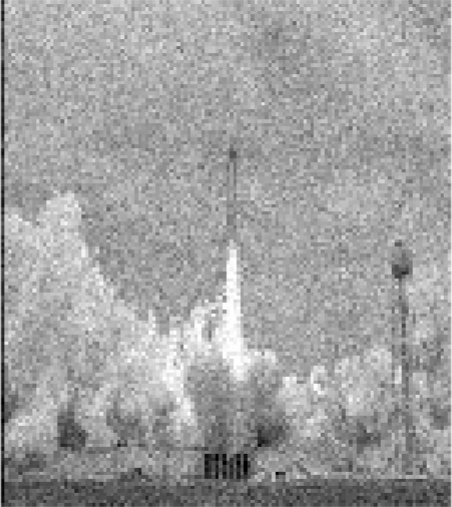}{%
            SSIM = 0.33196\\
            PSNR = 22.45 dB%
        }%
    }

    \vspace{6pt}

    \makebox[\linewidth][c]{%
        \mriresult{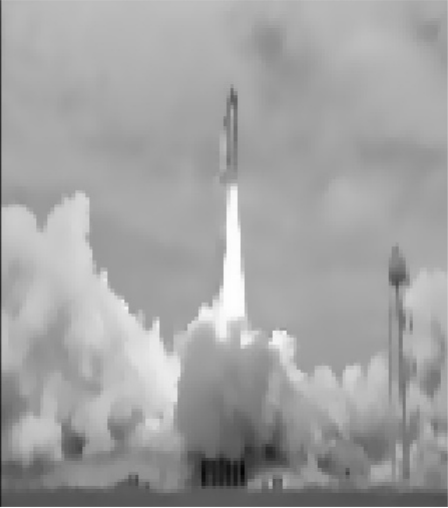}{%
            SSIM = 0.94148\\
            PSNR = 34.86 dB\\
            $\rho = 1.299$%
        }\hfill
        \mriresult{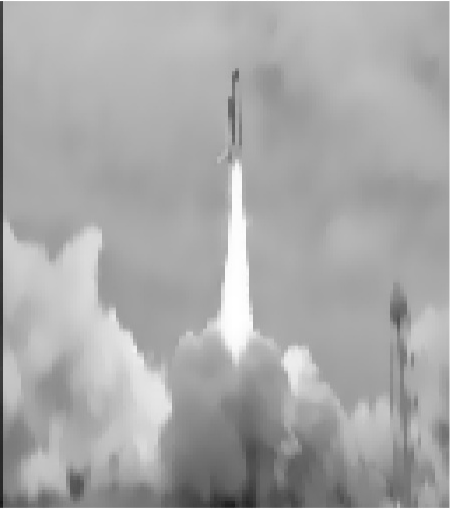}{%
            SSIM = 0.94148\\
            PSNR = 34.86 dB\\
            $\rho = 1.299$%
        }\hfill
        \mriresult{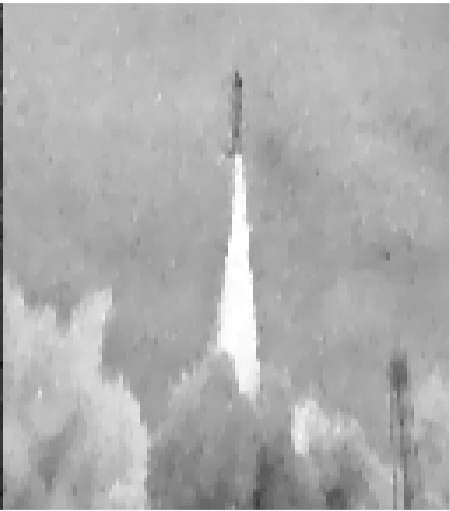}{%
            SSIM = 0.88457\\
            PSNR = 32.13 dB\\
            $\rho = 1.03$%
        }\hfill
        \mriresult{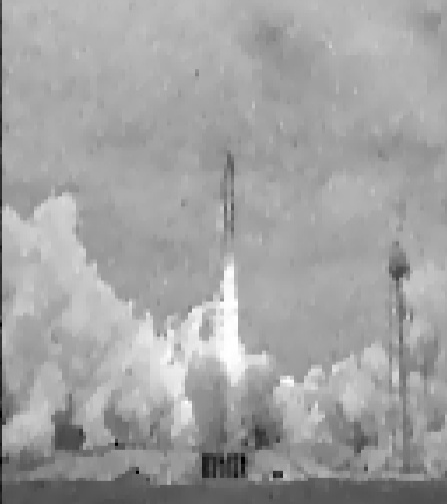}{%
            SSIM = 0.88457\\
            PSNR = 32.13 dB\\
            $\rho = 1.03$%
        }%
    }

    \caption{Representative frames from the 3D video. The first two columns correspond to the first experiment, while the last two columns correspond to the second experiment with a more degraded input. The $z$-values indicate the specific frames shown. The first row presents the ground-truth frames, the second row presents the blurred and noisy observations with their SSIM and
    PSNR values, and the third row presents the reconstructed frames with their SSIM, PSNR, and kernel error ratio~($\rho$) values.}

    \label{fig:shuttle}
\end{figure}

\bibliographystyle{plain}
\bibliography{references}
\end{document}